\documentclass[a4paper,reqno,12pt]{amsart}

\usepackage{amsmath}
\usepackage{amsthm}
\usepackage{amssymb}
\usepackage{latexsym}
\usepackage{graphicx,color}
\usepackage{booktabs}

\numberwithin{equation}{section}
\numberwithin{figure}{section}
\numberwithin{table}{section}

\newtheorem{thm}{Theorem}[section]

\newtheorem{lem}[thm]{Lemma}
\newtheorem{cor}[thm]{Corollary}

\theoremstyle{definition}

\theoremstyle{remark}

\allowdisplaybreaks[3]

\newcommand{\al}{\alpha}

\newcommand{\ga}{\gamma}

\newcommand{\e}{\varepsilon}
\newcommand{\fy}{\varphi}

\newcommand{\p}{\partial}

\newcommand{\I}{\infty}
\newcommand{\Sc}[1]{\mathcal{#1}}

\newcommand{\Bo}[1]{\mathbb{#1}}
\newcommand{\R}{\Bo{R}}
\newcommand{\T}{\Bo{T}}

\newcommand{\lec}{\lesssim}
\newcommand{\gec}{\gtrsim}

\newcommand{\shugo}[1]{\{ #1\}}
\newcommand{\Shugo}[2]{\big\{ \, #1 \, \big| \, #2 \, \big\}}

\newcommand{\eq}[2]{\begin{equation} \label{#1} \begin{split} #2 \end{split} \end{equation}}
\newcommand{\eqq}[1]{\begin{equation*} \begin{split} #1 \end{split} \end{equation*}}
\newcommand{\mat}[1]{\begin{smallmatrix} #1 \end{smallmatrix}}
\newcommand{\norm}[2]{\big\| #1 \big\| _{#2}}
\newcommand{\tnorm}[2]{\| #1 \| _{#2}}
\newcommand{\hx}{\hspace{10pt}}

\newcommand{\eqs}[1]{\begin{gather*} #1 \end{gather*}}

\title[Periodic nonlinear Schr\"odinger equation on torus]{Remark on the periodic mass critical nonlinear Schr\"odinger equation}
\author[Nobu Kishimoto]{Nobu Kishimoto}
\address{Department of Mathematics, Kyoto University, Kyoto 606-8502, Japan}
\email{n-kishi@math.kyoto-u.ac.jp}

\thanks{This work was partially supported by Grant-in-Aid for Scientific Research 23840022.}

\begin{document}

\begin{abstract}
We consider the mass critical NLS on $\T$ and $\T ^2$.
In the $\R ^d$ case the Strichartz estimates enable us to show well-posedness of the IVP in $L^2$ (at least for small data) via the Picard iteration method.
However, counterexamples to the $L^6$ Strichartz on $\T$ and the $L^4$ Strichartz on $\T ^2$ were given by Bourgain (1993) and Takaoka-Tzvetkov (2001), respectively, which means that the Strichartz spaces are not suitable for iteration in these problems.
In this note, we show a slightly stronger result, namely, that the IVP on $\T$ and $\T^2$ cannot have a smooth data-to-solution map in $L^2$ even for small initial data.
% Therefore, the standard iteration argument seems to fail in the periodic setting.
The same results are also obtained for most of the two dimensional irrational tori.
\end{abstract}

\maketitle

%%%%%%%%%%%%%%%%%%%%%
%%%%%%%%%%%%%%%%%%%%%
%%%%%%%%%%%%%%%%%%%%%

\section{Introduction}
We consider the initial value problem of the mass critical nonlinear Schr\"odinger equation with periodic boundary condition:
\begin{equation}\label{NLS}
\left\{
\begin{array}{@{\,}r@{\;}l}
i\p _tu+\Delta u&=\mu |u|^{4/d}u,\qquad (t,x)\in \R \times \T ^d,\\
u(0,x)&=u_0(x)\quad \in L^2(\T ^d).
\end{array}
\right.
\end{equation}
Here, $\T ^d:=\R ^d/(2\pi \Bo{Z})^d$ denotes the $d$ dimensional torus and $\mu =\pm 1$.

Concerning the initial value problem in $H^s(\T ^d)$ in the case $d=1,2$, Bourgain~\cite{B93-1} established the local well-posedness for $s>0$ via the Picard iteration method, and Christ, Colliander, Tao~\cite{CCT03-2} showed the ill-posedness for $s<0$ in the sense that the data-to-solution map is not continuous as a map from $H^s(\T)$ to even the space of distributions $(C^\I (\T))^*$ if $s<0$.
However, there is no result on the well-posedness of \eqref{NLS} in $L^2$ (neither positive nor negative), even for small data.

This difficulty comes from the lack of the periodic Strichartz estimate in the critical case.\footnote{The energy critical defocusing NLS on $\T^3$ is known to be globally well-posed in the critical space $H^1(\T^3)$; see \cite{HTT10,IP11}.
In that case, trilinear Strichartz-type estimates with no loss of regularity (with respect to scaling) are available and play a crucial role in the proof, which is based on the Picard iteration.}
In fact, some counterexamples were given in the case of $d=1,2$ (namely, in the case that the nonlinearity $|u|^{4/d}u$ is algebraic) as follows. 
\begin{thm}[Bourgain~\cite{B93-1}, Takaoka-Tzvetkov~\cite{TT01}]\label{thm:counterex}
The estimates
\eq{counterex1}{\norm{e^{it\p _x^2}\phi}{L^6_{t,x}(\T \times \T)}\lec \norm{\phi}{L^2(\T)}}
and
\eq{counterex2}{\norm{e^{it\Delta}\phi}{L^4_{t,x}(\T \times \T^2)}\lec \norm{\phi}{L^2(\T^2)}}
do not hold.
\end{thm}

In the case of $\R^d$, the corresponding Strichartz estimates hold and play an essential role in well-posedness theory at the critical regularity.
In fact, one can easily show the small-data local well-posedness with a smooth data-to-solution map by the iteration argument using these Strichartz estimates.
This is also true even for the 2d semi-periodic case $\R \times \T$~\cite{TT01}.
On $\T$ or $\T^2$, however, the above theorem suggests that the Strichartz space $L^{2+4/d}_{t,x}(\T ^{1+d})$ is no longer appropriate for the resolution space when we try to apply the iteration.
Then, a natural question is whether or not there is any other space suitable for iteration.

Firstly, we will show that for $d=1,2$ the first nonlinear term in the Picard iteration scheme is not bounded in $L^2$.
\begin{thm}\label{thm:notbounded}
Let $d=1,2$ and $N,m\in \Bo{N}$.
Define $\phi _{m,N}\in L^2(\T^d)$ by
\[ \phi _{m,N}(x):=N^{-\frac{d}{2}}\sum _{k\in \Bo{Z}_N^d}e^{imk\cdot x},\]
where $\Bo{Z}_N^d:=\Shugo{(k_1,\dots ,k_d)\in \Bo{Z}^d}{|k_j|\le N,~1\le j\le d}$.
Then, we have $\tnorm{\phi _{m,N}}{L^2(\T^d)}\sim 1$ and
\[ \norm{A[\phi _{m,N}](\frac{2\pi}{m^2})}{L^2(\T^d)}\gec \frac{1}{m^2}\log N,\]
where
\[ A[\phi ](t):=-i\mu \int _0^t e^{i(t-t')\Delta}\Big[ |e^{it'\Delta}\phi |^{\frac{4}{d}}e^{it'\Delta}\phi \Big] \,dt'.\]
\end{thm}

The main result in this note is obtained as a corollary of the above theorem.
This may be just a small but the first step toward the expected goal of extending local well-posedness to $L^2$. 
\begin{cor}\label{cor:notsmooth}
Let $d=1,2$ and $T,r>0$ be arbitrarily small positive constants.
Assume that the data-to-solution map $S:u_0\mapsto u(\cdot)$ associated with \eqref{NLS} on smooth data extends continuously to a map from the closed ball in $L^2(\T^d)$ of radius $r$ centered at the origin into $C([0,T];L^2(\T^d))$.
Then, this map will not be $C^5$ (resp. $C^3$) at the origin when $d=1$ (resp. $d=2$).
\end{cor}

Results of this type were first mentioned by Bourgain~\cite{B97} in the context of the KdV and the modified KdV initial value problems on $\R$ and $\T$.
To prove Corollary~\ref{cor:notsmooth}, we assume it not to hold.
Then, for 1d, the map $\phi \mapsto D^{5}S[\phi,\dots ,\phi ](0)=30A[\phi ]$ from $L^2(\T )$ to $C([0,T];L^2(\T ))$ would be continuous under the assumption of $C^5$.
This contradicts to Theorem~\ref{thm:notbounded} when we choose $m$ sufficiently large so that $\frac{2\pi}{m^2}\le T$.
The same argument is applicable to 2d.

Our result suggests that in the purely periodic setting the standard iteration argument, which would naturally give an analytic data-to-solution map, should fail to work for the mass critical NLS in $L^2$, even for small data.
This is a big difference from the nonperiodic case.

We next take the spatial period into account and consider the torus $\T^d_\al:=\R ^d/(\al _1\Bo{Z}\times \cdots \times \al _d\Bo{Z})$ with a general period $\al =(\al _1,\dots ,\al _d)\in \R _+^d$.
Such consideration makes no sense when $d=1$, since a scaling argument allows us to normalize one component of $\al$. 
In two or higher dimensions, however, \emph{rationality} of the torus may be of interest.
We say the 2d torus $\T^2_\al$ is rational (\mbox{resp.} irrational) if the ratio $\al _2/\al _1$ is (\mbox{resp.} is not) rational.
We refer to \cite{B07,CW10} for the Strichartz estimates on  general tori (both rational and irrational), which in general require a substantial amount of additional regularity.

It turns out that our argument for $d=2$ with a slight modification is applied not only to all of the rational tori but also to most of the irrational tori.
To our knowledge, this is the first result concerning 2d irrational tori at the $L^2$ regularity.
\begin{thm}\label{thm:iratori}
Let $\ga >0$ and $\T^2_\ga:=\R^2/(2\pi \Bo{Z}\times 2\pi\ga \Bo{Z})$.
Suppose that $\gamma$ satisfies the following:
\eq{dc}{\text{For any $\e >0$, there exist $p,q\in \Bo{N}$ such that $\Big| q^2-\frac{p^2}{\ga ^2}\Big| <\e$.}}
Then, the $L^4$ Strichartz estimate
\eqq{\norm{e^{it\Delta}\phi}{L^4_{t,x}([0,T]\times \T_\ga ^2)}\lec \norm{\phi}{L^2(\T _\ga ^2)}}
does not hold for any $T>0$.
Moreover, the statement in Corollary~\ref{cor:notsmooth} with $d=2$ and $\T^2$ replaced by $\T^2_\ga$ holds.
Furthermore, the set of all $\ga >0$ not satisfying \eqref{dc} is of Lebesgue measure zero.
\end{thm}
The condition \eqref{dc} means that $\ga$ can be well approximated by a rational $\frac{p}{q}$.
For functions on $\T^2_\ga$ we consider the Fourier coefficients on the rescaled lattice $\Bo{Z}\times \frac{1}{\ga}\Bo{Z}$.
Then, if \eqref{dc} is true, the sub-lattice $q\Bo{Z}\times \frac{p}{\ga}\Bo{Z}\subset \Bo{Z}\times \frac{1}{\ga}\Bo{Z}$ is close enough to a regular lattice $(q\Bo{Z})^2$ so that a similar argument to the case of $\T^2$ can be applied.
It seems, however, that the restriction \eqref{dc} is just a technical one.

Related to the $L^6$ Strichartz estimate \eqref{counterex1}, a similar estimate for the Airy equation,
\eqq{\norm{e^{-t\p _x^3}\phi}{L^6_{t,x}(\T \times \T )}\lec \norm{\phi}{L^2(\T )},}
has been attracting attention.
This estimate itself is open so far, while Bourgain~\cite{B93-2} proved it with $\e$-loss of regularity and conjectured that it would be true.

This note is organized as follows.
In the next section, we give a proof of Theorem~\ref{thm:notbounded}.
As a by-product of the proof, in section~3 we show Theorem~\ref{thm:counterex} in a different way without using the Weyl sum approach.
These proofs are modified in Section~\ref{sec:iratori} to verify Theorem~\ref{thm:iratori}.
In the last section, we make some remarks on the $L^6$ Strichartz estimate for the Airy equation.

% \bigskip
\section{Proof of Theorem~\ref{thm:notbounded}}\label{sec:proof}

First, we give a proof for the 2d case.
\begin{proof}[Proof of Theorem~\ref{thm:notbounded}, $d=2$]
Since
\[ e^{it\Delta}\phi _{m,N}=N^{-1}\sum _{k\in \Bo{Z}_N^2}e^{-im^2|k|^2t}e^{imk\cdot x},\]
we have
\eqq{A[\phi _{m,N}](t,x)&=cN^{-3}\sum _{k\in \Bo{Z}_{3N}^2}e^{imk\cdot x}\int _0^te^{-im^2|k|^2(t-t')}\sum _{\mat{k_1,k_2,k_3\in \Bo{Z}_N^2\\ k_1-k_2+k_3=k}}e^{-im^2(|k_1|^2-|k_2|^2+|k_3|^2)t'}\,dt'.}
Therefore, for $k\in \Bo{Z}_{3N}^2$,
\eqq{\hat{A}[\phi _{m,N}](\frac{2\pi}{m^2},mk)&=cN^{-3}\sum _{\mat{k_1,k_2,k_3\in \Bo{Z}_N^2\\ k_1-k_2+k_3=k}}\int _0^{2\pi/m^2}e^{-im^2(|k_1|^2-|k_2|^2+|k_3|^2-|k|^2)t'}\,dt'\\
&=cm^{-2}N^{-3}\# \Gamma (k),}
where the set of all \emph{resonant} interactions $\Gamma (k)$ is defined as
\eq{def:gamma}{\Gamma (k):=\Shugo{(k_1,k_2,k_3)\in (\Bo{Z}^2_{N})^3}{k_1-k_2+k_3=k,\,|k_1|^2-|k_2|^2+|k_3|^2=|k|^2}.}
This kind of set was previously observed in \cite{CKSTT10}.
In particular, the conditions for $(k_1,k_2,k_3)$ to be in $\Gamma (k)$ is equivalent to the condition that four segments $\overline{kk_1}$, $\overline{k_1k_2}$, $\overline{k_2k_3}$, $\overline{k_3k}$ form a rectangle (possibly degenerate); see \cite{CKSTT10}, section~2.2 for details.

It suffices to prove
\eq{counting}{\# \Gamma (k)\gec N^2\log N\quad \text{for $k\in \Bo{Z}^2_{N/2}$}.}
By translation, we may assume $k=0$.
The estimate \eqref{counting} then follows from the next lemma.
\end{proof}

\begin{lem}\label{lem:counting}
We have
\[ \# \Shugo{(k_1,k_3)\in (\Bo{Z}_N^2)^2}{k_1\cdot k_3=0}\gec N^2\log N.\]
\end{lem}
\begin{proof}
It suffices to count the number of elements of
\[ \Shugo{(p,q,r,s)\in \Bo{Z}^4}{0<q\le p\le N,\,0<s\le N, pr+qs=0},\]
but this is exactly equal to\footnote{There is a one-to-one correspondence between all the possible `direction' of vectors in $\Bo{Z}_+^2$ and all $(p,q)\in \Bo{Z}_+^2$ with $p$ co-prime to $q$.
Also, there are exactly $[N/\max \shugo{p,\,q}]$ points in $\Bo{Z}_{N}^2\cap \Bo{Z}_+^2$ for each direction $(p,q)$.}
\eqq{\sum _{\mat{0<q\le p\le N\\ \gcd (p,q)=1}}\Big[ \frac{N}{p}\Big] ^2=\sum _{p=1}^N\Big[ \frac{N}{p}\Big] ^2\varphi (p)\sim N^2\sum _{p=1}^N\frac{\varphi (p)}{p^2},}
where $[a]$ denotes the greatest integer not greater than $a$ and $\varphi(p)$ is Euler's totient function (\mbox{i.e.} the number of positive integers not greater than and relatively prime to $p$).
Recalling the identity $\sum _{d|n}\varphi (d)=n$,\footnote{
This identity follows from the fact that $\# \Shugo{k}{1\!\le \!k\!\le \!n,\, \gcd(k,n)=\frac{n}{d}}=\phi(d)$ for each positive divisor $d$ of $n$, which is a consequence of the equivalence $\gcd(m,n)=p~\Leftrightarrow ~\gcd(m/p,n/p)=1$.}
we have\footnote{Indeed, we have $\lim\limits _{N\to \I} \frac{1}{\log N}\sum\limits _{p=1}^{N}\frac{\varphi (p)}{p^2}=\frac{6}{\pi ^2}$.
This is the limiting case of the identity $\sum\limits _{p=1}^\I \frac{\varphi (p)}{p^s}=\frac{\zeta (s-1)}{\zeta (s)}$ for $s>2$, where $\zeta (s):=\sum\limits _{p=1}^\I \frac{1}{p^s}$.}
\eqq{\log (N+1)&<\sum _{n=1}^{N}\frac{1}{n}=\sum _{n=1}^N\frac{1}{n^2}\sum _{\mat{1\le p,l\le n\\ pl=n}}\varphi (p)\le \sum _{p,l=1}^N\frac{\varphi (p)}{p^2l^2}<\frac{\pi ^2}{6}\sum _{p=1}^N\frac{\varphi (p)}{p^2},}
obtaining the claim.
\end{proof}

Proof for the 1d quintic case is reduced to the above 2d result.
\begin{proof}[Proof of Theorem~\ref{thm:notbounded}, $d=1$]
Similarly to the case $d=2$, we have
\eqq{\hat{A}[\phi _{m,N}](\frac{2\pi}{m^2},mk)&=cN^{-\frac{5}{2}}\sum _{\mat{k_1,k_2,k_3,k_4,k_5\in \Bo{Z}_N\\ k_1-k_2+k_3-k_4+k_5=k}}\int _0^{2\pi/m^2}e^{-im^2(k_1^2-k_2^2+k_3^2-k_4^2+k_5^2-k^2)t'}\,dt'\\
&=cm^{-2}N^{-\frac{5}{2}}\# \Gamma '(k),}
where
\[ \Gamma '(k):=\Shugo{(k_1,k_2,k_3,k_4,k_5)\in \Bo{Z}_{N}^5}{k_1-k_2+k_3-k_4+k_5=k,\,k_1^2-k_2^2+k_3^2-k_4^2+k_5^2=k^2},\]
and it suffices to show
\eq{counting1}{\# \Gamma '(k)\gec N^2\log N\quad \text{for $k\in \Bo{Z}_{N/2}$}.}
Putting
\[ l_1=\frac{k_1+k_2}{2},\, n_1=\frac{k_1-k_2}{2},\,l_2=\frac{k_3+k_4}{2},\, n_2=\frac{k_3-k_4}{2},\,l_3=\frac{k_5+k}{2},\, n_3=\frac{k_5-k}{2},\]
we reduce \eqref{counting1} to
\eqq{\# \Shugo{(l_1,\dots ,n_3)\in \Bo{Z}_N^6}{n_1+n_2+n_3=0,\,l_1n_1+l_2n_2+l_3n_3=0,\,l_3=n_3+k}\gec N^2\log N,}
which is further simplified to
\eqq{\# \Shugo{(l_1,l_2,n_1,n_2)\in \Bo{Z}_N^4}{(l_1+n_1+n_2-k)n_1+(l_2+n_1+n_2-k)n_2=0}\gec N^2\log N.}
Then, denoting $l_j':=l_j+n_1+n_2-k$ ($j=1,2$), it suffices to show
\eqq{\# \Shugo{(l_1',l_2',n_1,n_2)\in \Bo{Z}_N^4}{l_1'n_1+l_2'n_2=0}\gec N^2\log N.}
This is nothing but Lemma~\ref{lem:counting}.
\end{proof}

% \bigskip
\section{Proof of Theorem~\ref{thm:counterex}}\label{sec:counterex}

Here, we give another proof of the failure of \eqref{counterex2} and \eqref{counterex1}.
We use $\phi =\phi _{1,N}$ for the sequence of data breaking these estimates.
There is no difference between 1d and 2d, so we focus on the 2d estimate \eqref{counterex2}.
The exact calculation shows that
\eqq{&\norm{e^{it\Delta}\phi _{1,N}}{L^4_{t,x}(\T^3)}^4=N^{-4}\sum _{k_1,k_2,k_3,k_4\in \Bo{Z}_N^2}\int _{\T^3}e^{-i(|k_1|^2-|k_2|^2+|k_3|^2-|k_4|^2)t}e^{i(k_1-k_2+k_3-k_4)\cdot x}\,dx\,dt\\
&=(2\pi )^3N^{-4}\# \Shugo{(k_1,k_2,k_3,k_4)\in (\Bo{Z}_N^2)^4}{k_1-k_2+k_3-k_4=0,\,|k_1|^2-|k_2|^2+|k_3|^2-|k_4|^2=0}\\
&\ge (2\pi )^3N^{-4}\sum _{k_4\in \Bo{Z}_{N/2}^2}\# \Gamma (k_4)\\
&\gec \log N,}
where $\Gamma(k)$ is as in \eqref{def:gamma} and we have used \eqref{counting} at the last inequality.
Since $\tnorm{\phi _{1,N}}{L^2(\T^2)}\sim 1$, the estimate \eqref{counterex2} fails for sufficiently large $N$.
This choice of initial data is the same as \cite{TT01}, but we obtain the lower bound $(\log N)^{1/4}$, which is better than $(\log \log N)^{1/4}$ in \cite {TT01}.
For \eqref{counterex1} we can show the same lower bound $(\log N)^{1/6}$ as \cite{B93-1}.

We can also disprove
\eqq{\norm{e^{it\Delta}\phi}{L^{2+d/2}_{t,x}([0,T]\times \T^d)}\lec \norm{\phi}{L^2(\T^d)}}
for arbitrary $T>0$ ($d=1,2$) by choosing $m\in \Bo{N}$ so that $\frac{2\pi}{m^2}\le T$ and repeating the above argument with $\phi =\phi _{m,N}$.

In the remainder of this section, let us consider the $L^4$ estimate on $\T \times \T^3$:
\eq{3dL4}{\norm{e^{it\Delta}\phi}{L^4_{t,x}(\T \times \T ^3)}\lec N^{\frac{1}{4}}\norm{\phi}{L^2(\T ^3)}}
for all $\phi \in L^2$ with $\hat{\phi}(k)\equiv 0$ if $|k|>N$.
Bourgain~\cite{B93-1} proved \eqref{3dL4} up to an $\e$ loss of regularity, but \eqref{3dL4} itself has been open.
We will see that the same sequence of data $\phi _{1,N}$ does not break it any longer, so one cannot disprove \eqref{3dL4} by just adapting 1d or 2d counterexample to the 3d situation.

Recalling that $\phi _{1,N}(x):= \frac{1}{N^{3/2}}\sum\limits _{k\in \Bo{Z}_N^3}e^{ik\cdot x}$, we have
\eqs{\norm{e^{it\Delta}\phi _{1,N}}{L^4_{t,x}(\T^4)}^4=(2\pi )^4N^{-6}\sum _{k_4\in \Bo{Z}_{N}^3}\# \Gamma ''(k_4),\\
\Gamma ''(k_4):=\Shugo{(k_1,k_2,k_3)\in (\Bo{Z}_N^3)^3}{k_1-k_2+k_3=k_4,\,|k_1|^2-|k_2|^2+|k_3|^2=|k_4|^2}.}
We will show $\# \Gamma ''(k)\lec N^4$ for any $k\in \Bo{Z}^3_N$, which verifies that $\phi _{1,N}$ obeys \eqref{3dL4} for all $N$.
Similarly to the 2d case we observed in the previous section, the rectangular structure of resonant frequencies implies that it suffices to prove $\# \Gamma'' (0)\lec N^4$, and that $\# \Gamma ''(0)\le \# \Shugo{(k_1,k_3)\in (\Bo{Z}^3_N)^2}{k_1\cdot k_3=0}$.

In the case that either $k_1$ or $k_3$ is located on an axis, we easily obtain a bound of $O(N^3)$.
For instance, if $k_1\in \Bo{Z}_N\times \shugo{(0,0)}$, then $k_3$ must be in $\shugo{0}\times \Bo{Z}_N^2$.

Hence, without loss of generality, we count the number of $(k_1,k_3)$ such that $k_1\in \Shugo{(x,y,z)\in \Bo{Z}^3}{0<z\le y\le x\le N}$.
For any `direction' $(a,b,c)$ ($0<c\le b\le a\le N$ and $\gcd (a,b,c)=1$), there are exactly $[\frac{N}{a}]$ choices of $k_1$ facing in that direction, and $k_3$ must be perpendicular to that direction for such a $k_1$.
Therefore, the number of possible $(k_1,k_3)$ is bounded by
\eq{est11}{\sum _{\mat{0<c\le b\le a\le N\\ \gcd (a,b,c)=1}}\frac{N}{a}\# \Big( \Sc{P}_{(a,b,c)}\cap \Bo{Z}_N^3\Big) ,}
where $\Sc{P}_{(a,b,c)}$ is the plane in $\R^3$ including the origin and perpendicular to $(a,b,c)$.
Then, it suffices to prove $\eqref{est11}=O(N^4)$.

Fix $(a,b,c)$ satisfying the above condition and set $n:=\gcd (a,b)$, $(a,b)=(na',nb')$.
Then, 
\eqq{\# \Big( \Sc{P}_{(a,b,c)}\cap \Bo{Z}_N^3\cap \Shugo{(x,y,z)\in \Bo{Z}^3}{z=0}\Big) \sim \frac{N}{a'}.}
Since $\gcd (c,n)=1$, the plane $\Shugo{(x,y,z)\in \Bo{Z}^3}{z=l}$ is possible to intersect $\Sc{P}_{(a,b,c)}\cap \Bo{Z}_N^3$ if and only if $l$ is a multiple of $n$.
Therefore, we have
\eqq{\# \Big( \Sc{P}_{(a,b,c)}\cap \Bo{Z}_N^3\Big) \sim \frac{N}{a'}\cdot \frac{N}{n}=\frac{N^2}{a},}
and then
\eq{est12}{&\sum _{\mat{0<c\le b\le a\le N\\ \gcd (a,b,c)=1}}\frac{N}{a}\# \Big( \Sc{P}_{(a,b,c)}\cap \Bo{Z}_N^3\Big) \sim N^3\sum _{a=1}^N\frac{1}{a^2}\# \Shugo{(b,c)}{0<c\le b\le a,\,\gcd (a,b,c)=1}.}

We also have
\eqq{&\# \Shugo{(b,c)}{0<c\le b\le a,\,\gcd (a,b,c)=1}=\sum _{\mat{1\le n\le a\\ n|a}}\sum _{\mat{1\le b\le a\\ \gcd (b,a)=n}}\sum _{\mat{1\le c\le b\\ \gcd (c,n)=1}}1\\
&=\sum _{\mat{1\le n\le a\\ n|a}}\sum _{\mat{1\le b'\le a/n\\ \gcd (b',a/n)=1}}b'\fy (n)\le \sum _{\mat{1\le n\le a\\ n|a}}\sum _{\mat{1\le b'\le a/n\\ \gcd (b',a/n)=1}}\frac{a}{n}\cdot n=a\sum _{\mat{1\le n\le a\\ n|a}}\fy (\frac{a}{n})=a^2,}
where at the last equality we have used that $\sum\limits _{n|a}\fy (\frac{a}{n})=\sum\limits _{n|a}\fy(n)=a$.
Plugging this into \eqref{est12} we obtain $\eqref{est11}\lec N^4$, as desired.

% \bigskip
\section{The case of general tori}\label{sec:iratori}

Let us give a proof of Theorem~\ref{thm:iratori} for a general 2d torus $\T _\ga ^2:=\R^2/(2\pi \Bo{Z}\times 2\pi \ga \Bo{Z})$ ($\ga >0$).
Under the condition \eqref{dc}, for any $N\gg 1$ we can choose $p,q\in \Bo{N}$ such that
\eqq{\Big| q^2-\frac{p^2}{\ga ^2}\Big| <\frac{1}{N^2},\qquad q>N.}
Indeed, it is trivial if $\gamma \in \Bo{Q}$, and otherwise it suffices to take $p,q$ given in \eqref{dc} with  
\[ \e =\min \shugo{\frac{1}{N^2},\,\min _{\mat{p,q\in \Bo{N},\, q\le N}}\Big| q^2-\frac{p^2}{\ga ^2}\Big|}>0.\]
Using such $(p,q)$ we define $\phi _N:\T^2_\ga \to \Bo{C}$ by
\[ \phi _N(x,y):=\frac{1}{N}\sum _{k=(k_x,k_y)\in \Bo{Z}_N^2}e^{i(qk_xx+\frac{p}{\ga}k_yy)},\]
which has the size of $O(1)$ in $L^2(\T^2_\ga)$.
Similarly to Section~\ref{sec:proof}, we have
\eqq{\hat{A}[\phi _{N}](t,qk_x,\frac{p}{\ga}k_y)&=ce^{-i(q^2k_x^2+\frac{p^2}{\ga ^2}k_y^2)t}N^{-3}\sum _{\mat{k_1,k_2,k_3\in \Bo{Z}_N^2\\ k_1-k_2+k_3=k}}\int _0^te^{-i\Phi t'}\,dt'}
for $k=(k_x,k_y)\in \Bo{Z}^2$, where
\eqq{\Phi :=&\,q^2(k_{1,x}^2-k_{2,x}^2+k_{3,x}^2-k_x^2)+\frac{p^2}{\ga ^2}(k_{1,y}^2-k_{2,y}^2+k_{3,y}^2-k_y^2)\\
=&\,q^2(|k_1|^2-|k_2|^2+|k_3|^2-|k|^2)-\Big( q^2-\frac{p^2}{\ga ^2}\Big) (k_{1,y}^2-k_{2,y}^2+k_{3,y}^2-k_y^2).}
We split $A[\phi _N]$ into the resonant part and the non-resonant part:
\eqq{&\hat{A}[\phi _{N}](t,qk_x,\frac{p}{\ga}k_y)\\
&=ce^{-i(q^2k_x^2+\frac{p^2}{\ga ^2}k_y^2)t}N^{-3}\bigg( \sum _{|k_1|^2-|k_2|^2+|k_3|^2=|k|^2}+\sum _{|k_1|^2-|k_2|^2+|k_3|^2\neq |k|^2}\bigg) \int _0^te^{-i\Phi t'}\,dt'\\
&=:\hat{A}_{res}[\phi _{N}](t,qk_x,\frac{p}{\ga}k_y)+\hat{A}_{nonres}[\phi _{N}](t,qk_x,\frac{p}{\ga}k_y).}

For the resonant part, we have
\eqq{|\Phi |=\Big| \Big( q^2-\frac{p^2}{\ga ^2}\Big) (k_{1,y}^2-k_{2,y}^2+k_{3,y}^2-k_y^2)\Big| \lec \frac{1}{N^2}\cdot N^2=1,}
which implies $\Re \int _0^te^{-i\Phi t'}dt'\ge \frac{t}{2}$ for any $0<t\ll 1$.
Therefore, from \eqref{counting} we have
\eqq{|\hat{A}_{res}[\phi _N](t,qk_x,\frac{p}{\ga}k_y)|\gec tN^{-1}\log N}
for any $k\in \Bo{Z}^2_{N/2}$ and $0<t\ll 1$.
For the non-resonant part, it holds that
\eqq{|\Phi |&\ge |q^2(|k_1|^2-|k_2|^2+|k_3|^2-|k|^2)|-\Big| \Big( q^2-\frac{p^2}{\ga ^2}\Big) (k_{1,y}^2-k_{2,y}^2+k_{3,y}^2-k_y^2)\Big| \\
&\ge q^2-O(1)\gec N^2,}
which implies $|\int _0^te^{-i\Phi t'}dt'|\lec N^{-2}$ for any $t\in \R$.
Therefore,
\eqq{|\hat{A}_{nonres}[\phi _N](t,qk_x,\frac{p}{\ga}k_y)|\lec N^{-3}\cdot N^4\cdot N^{-2}=N^{-1}.}
Consequently, we have
\eqq{\norm{A[\phi _N](t)}{L^2(\T^2_\ga )}\gec t\log N}
for any $(\log N)^{-1}\ll t\ll 1$, which shows that the map $\phi \mapsto A[\phi ](t)$ on $L^2$ is not continuous at the origin for any $0<t\ll 1$.
From this we deduce the same conclusion as Corollary~\ref{cor:notsmooth}.
The failure of the $L^4$ Strichartz estimate is shown in a similar manner; we refer to Section~\ref{sec:counterex} and omit the proof.

Finally, we observe the condition \eqref{dc}, which is equivalent to the following:
\eq{dc'}{\text{For any $\e >0$, there exist $p,q\in \Bo{N}$ such that $\Big| \ga -\frac{p}{q}\Big| <\frac{\e}{q^2}$.}}
This is satisfied if $\gamma \in \Bo{Q}$, so we may assume $\gamma \not\in \Bo{Q}$.
Recall the continued fraction expansion of $\gamma$ (\mbox{cf.} \cite{HWbook}, chapter X): There exists a unique sequence of positive integers $\shugo{a_n}_{n=1}^{\I}$ such that $\gamma$ is the limit of the sequence of finite continued fractions
\eq{pnqn}{[\ga ]+\frac{1}{a_1+\frac{1}{a_2+\frac{1}{\ddots +\frac{1}{a_n}}}},\qquad n=1,2,\dots}
% \eq{pnqn}{[\ga ]+\frac{1}{a_1+\dfrac{1}{a_2+\dfrac{1}{\ddots +\dfrac{1}{a_n}}}},\qquad n=1,2,\dots}
as $n\to \I$.
It is known (\cite{HWbook}, section~11.10) that \eqref{dc'} will be satisfied if the sequence $\shugo{a_n}$ corresponding to $\ga$ is unbounded,%
\footnote{
In fact, \eqref{dc'} holds if and only if $\shugo{a_n}$ is unbounded.
To see this, let $p_n,q_n\in \Bo{N}$ be such that $\frac{p_n}{q_n}$ is the irreducible fraction representation of \eqref{pnqn}.
Then, it holds that under the convention $q_0=1$

\centerline{$\tfrac{1}{q_n^2a_{n+1}}>|\gamma -\tfrac{p_n}{q_n}|>\tfrac{1}{q_n^2(a_{n+1}+2)}>\tfrac{1}{q_{n-1}^2(a_n+1)^2(a_{n+1}+2)}$}

\noindent for any $n$.
Moreover, the sequence $\shugo{\frac{p_n}{q_n}}$ is the best rational approximation of $\gamma$ in the sense that

\centerline{$|\gamma -\frac{p_n}{q_n}| =\min \Shugo{|\gamma -\frac{p}{q}|}{p,q\in \Bo{N},\,q\le q_n}$}

\noindent for any $n$ (see \cite{HWbook}, Theorem~181).
Hence, if $a_n\le M<\I$ for all $n$, we have $|\gamma -\frac{p}{q}|>\frac{1}{(M+1)^2(M+2)q^2}$ for any $p,q\in \Bo{N}$.
As a corollary, it turns out that no quadratic irrational $\ga$ (\mbox{i.e.} irrational root of a quadratic equation with integral coefficients) satisfy \eqref{dc}, since $\shugo{a_n}$ for such an irrational is periodic (see \cite{HWbook}, Theorem~177).
}
and that the set of $\ga$ for which $\shugo{a_n}$ is bounded is null.
Hence, almost every $\ga >0$ satisfies \eqref{dc}.

% \bigskip
\section{Note on the $L^6$ Strichartz estimate for the Airy equation}

The $L^6$ Strichartz estimate for the Airy equation,
\eq{KdVL6}{\norm{e^{-t\p _x^3}\phi}{L^6_{t,x}(\T \times \T )}\lec \norm{\phi}{L^2(\T )},}
is a challenging open problem proposed by Bourgain~\cite{B93-2}, who proved instead
\eq{KdVL6'}{\norm{e^{-t\p _x^3}\phi}{L^6_{t,x}(\T \times \T )}\lec _\e N^\e \norm{\phi}{L^2(\T )}}
for all $\phi \in L^2$ with $\hat{\phi}(k)\equiv 0$ if $|k|>N$.

Let us recall the proof of \eqref{KdVL6'}.
Let $\phi =\sum\limits _{k\in \Bo{Z}_N}a_ke^{ikx}$.
By the Cauchy-Schwarz inequality,
\eqq{\norm{e^{-t\p _x^3}\phi}{L^6_{t,x}(\T ^2)}^6&=\norm{\sum _{k_1,k_2,k_3\in \Bo{Z}_N}a_{k_1}a_{k_2}a_{k_3}e^{i(k_1^3+k_2^3+k_3^3)t}e^{i(k_1+k_2+k_3)x}}{L^2_{t,x}(\T ^2)}^2\\
&=\sum _{n,k\in \Bo{Z}}\Bigg| \sum _{\mat{k_1,k_2,k_3\in \Bo{Z}_N\\ k_1+k_2+k_3=k,\,k_1^3+k_2^3+k_3^3=n}}a_{k_1}a_{k_2}a_{k_3}\Bigg| ^2\\
&\le \sum _{k\in \Bo{Z}_N}\Bigg| 3\sum _{k_1\in \Bo{Z}_N}a_{k_1}a_{-k_1}a_k\Bigg| ^2\\
&\hx +\bigg( \sup _{\mat{n\in \Bo{Z}_{3N^3},\,k\in \Bo{Z}_{3N}\\ n\neq k^3}}\# \Gamma_{\text{Airy}} (n,k)\bigg) \sum _{n,k\in \Bo{Z}}\sum _{\mat{k_1,k_2,k_3\in \Bo{Z}_N\\ k_1+k_2+k_3=k,\,k_1^3+k_2^3+k_3^3=n}}|a_{k_1}|^2|a_{k_2}|^2|a_{k_3}|^2\\
&\le 9\norm{\phi}{L^2(\T )}^6+\norm{\phi}{L^2(\T )}^6\sup _{\mat{n\in \Bo{Z}_{3N^3},\,k\in \Bo{Z}_{3N}\\ n\neq k^3}}\# \Gamma_{\text{Airy}} (n,k),}
where
\eqq{\Gamma_{\text{Airy}} (n,k):=\Shugo{(k_1,k_2,k_3)\in \Bo{Z}_N^3}{k_1+k_2+k_3=k,\,k_1^3+k_2^3+k_3^3=n}.}
Some divisor-counting argument then yields the bound
\eqq{\# \Gamma_{\text{Airy}} (n,k)\lec e^{\frac{c\log N}{\log \log N}}\lec _\e N^\e \quad \text{for $|n|\lec N^3$ and $|k|\lec N$ \mbox{s.t.} $n\neq k^3$},}
which implies \eqref{KdVL6'}.

The above proof says that we could establish \eqref{KdVL6} if we had the uniform estimate
\eq{countingB}{\# \Gamma_{\text{Airy}} (n,k)\lec 1\quad \text{for $n,k\in \Bo{Z}$ \mbox{s.t.} $n\neq k^3$}.}
A similar estimate holds if we put some further restriction on the set $\Gamma_{\text{Airy}} (n,k)$.
In fact, Colliander et al.~\cite{CKSTT04} proved \eqref{countingB} with $\Gamma_{\text{Airy}} (n,k)$ replaced by
\eqq{\Gamma_{\text{Airy}} '(n,k):=\Shugo{(k_1,k_2,k_3)\in \Gamma_{\text{Airy}} (n,k)}{|k_{max}|\gg |k_{med}|\gg |k_{min}|,\,|k_{max}|\gg |k_{min}|^3},}
where $k_{max}$, $k_{med}$, $k_{min}$ are the maximum, the median, and the minimum among $k_1,k_2,k_3$, respectively,
and used it in \cite{CKSTT03} to prove the global well-posedness of the Korteweg-de Vries equation on $\T$ at the limiting regularity $H^{-1/2}$.

Unfortunately, \eqref{countingB} itself is false.\footnote{
The author could not find any article pointing out this fact.
}
More precisely, it holds that for each $N\gg 1$ there exists $k\in \Bo{Z}$ such that $|k|\lec N$ and
\eqq{\# \Gamma_{\text{Airy}} (\frac{k^3}{9},k)\gec \log N.}
To see this, we first notice that $k_1+k_2+k_3=k$ and $k_1^3+k_2^3+k_3^3=n$ imply 
\eqs{(k-k_1)+(k-k_2)+(k-k_3)=2k,\\
(k-k_1)(k-k_2)(k-k_3)=\frac{1}{3}\big\{ (k_1+k_2+k_3)^3-(k_1^3+k_2^3+k_3^3)\big\} =\frac{k^3-n}{3}.}
Putting $k_j':=\frac{1}{2}(k-k_j)$ ($j=1,2,3$), we consider sets of three integers $k_1',k_2',k_3'$  which have common sum and product.
Now, consider three rational numbers
\eqq{\frac{(x+1)^2}{x},\quad -\frac{x^2}{x+1},\quad -\frac{1}{x(x+1)}}
for some $x\in \Bo{N}$.
Note that the sum and the product of these three numbers are $3$ and $1$, respectively, independent of $x$.\footnote{This is actually true for three fractions $\frac{x^2}{yz},\frac{y^2}{zx},\frac{z^2}{xy}$ with $x,y,z\in \Bo{Z}$ satisfying $x+y+z=0$ and $xyz\neq 0$.
The above one is just a special case ($y=1$, $z=-x-1$) of it.}
Moreover, a different choice of $x$ gives a different set of three numbers, since these three fractions are all irreducible.
Hence, for an arbitrary $m\in \Bo{N}$ we find $m$ sets of three integers
\eqq{\big( \frac{(x+1)^2}{x}M,\quad -\frac{x^2}{x+1}M,\quad -\frac{1}{x(x+1)}M\big) ,\quad x\in \shugo{1,2,\dots ,m}}
with the common sum $3M$ and the common product $M^3$, where $M$ denotes the least common multiple\footnote{
Since $x+1$ is relatively prime to $x$, $M/x(x+1)\in \Bo{Z}$.
} %
of $1,2,\dots ,m+1$.
It is easily verified that
\eqq{\log M=\log \prod _{\mat{p:\text{prime}\\p\le m+1}}p^{\left[ \frac{\log (m+1)}{\log p}\right]}\le \pi (m+1)\log (m+1),}
where $\pi (n)$ is the number of prime numbers not greater than $n$.
From the prime number theorem\footnote{$\pi (n)\sim \frac{n}{\log n}$ ($n\gg 1$).} we obtain $\log M \lec m$ for large $m$.
Finally, from $2\cdot 3M=2k$ and $8\cdot M^3=\frac{1}{3}(k^3-n)$ we have $k=3M$ and $n=3M^3$, and we find $m$ sets of three integers
\eqq{\big( (3-\frac{2(x+1)^2}{x})M,\quad (3+\frac{2x^2}{x+1})M,\quad (3+\frac{2}{x(x+1)})M\big) ,\quad x\in \shugo{1,2,\dots ,m}}
with the common sum $3M$ and the common cubic sum $3M^3$.

We remark that the above observation seems not strong enough to disprove \eqref{KdVL6}, because the logarithmic growth is verified only for very few $(n,k)$.
Conversely, \eqref{KdVL6} could be established if we had
\eqq{\sum _{k\in \Bo{Z}}\bigg| \sum _{\mat{k_1,k_2,k_3\in \Bo{Z}_N\\ k_1+k_2+k_3=k,\,k_1^3+k_2^3+k_3^3=k^3/9}}a_{k_1}a_{k_2}a_{k_3}\bigg| ^2\lec \norm{\phi}{L^2(\T )}^6}
and
\eqq{\# \Gamma_{\text{Airy}} (n,k)\lec 1\quad \text{for $n,k\in \Bo{Z}$ \mbox{s.t.} $n\neq k^3$ and $n\neq \dfrac{k^3}{9}$}.}
It is not clear, however, whether both or either of them are true.

%%%%%%%%%%%%%%%%%%%%%%%%%%%%%%%%%%%
%%%%%%%%%%%%%%%%%%%%%%%%%%%%%%%%%%%
%%%%%%%%%%%%%%%%%%%%%%%%%%%%%%%%%%%

% \bigskip
% \bigskip

\end{document}